\documentclass[a4paper,11pt]{amsart}

\usepackage{amsmath,amscd,amssymb,mathdots,enumerate}
\usepackage[all]{xy}

\renewcommand{\mod}{\operatorname{mod}\nolimits}
\renewcommand{\Im}{\operatorname{Im}\nolimits}
\renewcommand{\dim}{\operatorname{dim}\nolimits}

\newcommand{\op}{{\operatorname{op}\nolimits}}
\newcommand{\Hom}{\operatorname{Hom}\nolimits}
\newcommand{\Ker}{\operatorname{Ker}\nolimits}
\newcommand{\Tor}{\operatorname{Tor}\nolimits}
\newcommand{\Ext}{\operatorname{Ext}\nolimits}
\newcommand{\HH}{\operatorname{HH}\nolimits}
\newcommand{\pdim}{\operatorname{pdim}\nolimits}
\newcommand{\idim}{\operatorname{idim}\nolimits}
\newcommand{\gldim}{\operatorname{gldim}\nolimits}
\newcommand{\res}{\operatorname{res}\nolimits}
\newcommand{\rad}{\operatorname{rad}\nolimits}
\newcommand{\rrad}{\mathfrak{r}}
\newcommand{\cQ}{{\mathcal Q}}
\newcommand{\mo}{{\mathfrak o}}
\newcommand{\mt}{{\mathfrak t}}

\newtheorem{lem}{Lemma}[section]
\newtheorem{prop}[lem]{Proposition}
\newtheorem{cor}[lem]{Corollary}
\newtheorem{thm}[lem]{Theorem}
\theoremstyle{definition}
\newtheorem{defin}[lem]{Definition}
\newtheorem{remark}[lem]{Remark}
\newtheorem{example}[lem]{Example}
\newtheorem*{remark*}{Remark}

\begin{document}

\topmargin 0cm
\oddsidemargin 0.5cm
\evensidemargin 0.5cm

\title[Hochschild cohomology, finiteness conditions and \dots]
{Hochschild cohomology, finiteness conditions and a generalisation of $d$-Koszul algebras}

\author[Jawad]{Ruaa Jawad}
\address{Ruaa Jawad\\
Preparation of Trained Technicians Institute, Middle Technical University \\
University of Baghdad Post Office, Al Jadriya, P.O. Box 47123\\
Baghdad, Iraq}
\email{ruaayousuf@gmail.com}
\author[Snashall]{Nicole Snashall}
\address{Nicole Snashall\\
School of Mathematics and Actuarial Science\\
University of Leicester \\
University Road \\
Leicester LE1 7RH \\
United Kingdom}
\email{njs5@le.ac.uk}

\thanks{This work formed part of the first author's PhD thesis at the University of Leicester, which was supported by The Higher Committee For Education Development in Iraq (HCED)}


\begin{abstract}
Given a finite-dimensional algebra $\Lambda$ and $A \geqslant 1$, we construct a new algebra $\tilde{\Lambda}_A$, called the stretched algebra, and relate the homological properties of $\Lambda$ and $\tilde{\Lambda}_A$. We investigate Hochschild cohomology and the finiteness condition {\bf (Fg)}, and use stratifying ideals to show that $\Lambda$ has {\bf (Fg)} if and only if $\tilde{\Lambda}_A$ has {\bf (Fg)}. We also consider projective resolutions and apply our results in the case where $\Lambda$ is a $d$-Koszul algebra for some $d \geqslant 2$.
\end{abstract}

\maketitle

\section*{Introduction}

Let $K$ be a field and let $\Lambda = K\cQ/I$ be a finite-dimensional
algebra where $I$ is an admissible ideal of $K\cQ$.
For each $A \geqslant 1$, we construct from $\Lambda$ a new algebra
$\tilde{\Lambda}_A$, called the stretched algebra.
The aim of the paper is to relate the homological properties of $\Lambda$ and $\tilde{\Lambda}_A$. In Section~\ref{sec:Fg}, the focus is on Hochschild cohomology and the finiteness condition {\bf (Fg)} of \cite{EHSST}, and in Section~\ref{sec:minprojres} we look at projective resolutions and apply the results to construct examples of stretched algebras.

Section~\ref{sec:Fg} studies the Hochschild cohomology of $\Lambda$ and the stretched algebra $\tilde{\Lambda}_A$. Our motivation here lies in the theory of support varieties.
For a group algebra of a finite group, Carlson introduced a powerful theory of support varieties of modules \cite{C1}, \cite{C2}.
Support varieties were extended to finite-dimensional algebras by Snashall and Solberg in \cite{SS}, using the Hochschild cohomology ring of the algebra.
And, under the finiteness condition {\bf (Fg)} of \cite{EHSST} (see Definition~\ref{defin:fg}), many of the properties known for the group situation were shown to have analogues in this more general setting.
Subsequently, the condition {\bf (Fg)} has been widely studied. Our intention is to use Nagase's result \cite[Proposition 6]{N} concerning {\bf (Fg)} and algebras with stratifying ideals.
In Theorem~\ref{thm:stratifying ideal} we give an idempotent element $\varepsilon$ of the stretched algebra $\tilde{\Lambda}_A$, proving that $\langle\varepsilon\rangle$ is a stratifying ideal in $\tilde{\Lambda}_A$. We then show in Corollary~\ref{cor:pdim} that the projective dimension of $\tilde{\Lambda}_A/\langle\varepsilon\rangle$ is 2 as a $\tilde{\Lambda}_A$-$\tilde{\Lambda}_A$-bimodule. Our main result is Theorem~\ref{thm:fg}, where we show that $\tilde{\Lambda}_A$ has {\bf (Fg)} if and only if $\Lambda$ has {\bf (Fg)}.

Section~\ref{sec:minprojres} considers projective resolutions. In Theorem~\ref{thm:projres}, we start with a minimal projective resolution of $\Lambda/\rrad$
as a right $\Lambda$-module, and explicitly describe a minimal projective
resolution of $\tilde{\Lambda}_A / \tilde{\rrad}_A$ as a right $\tilde{\Lambda}_A$-module, where $\rrad$ (resp.\ $\tilde{\rrad}_A$) denotes the Jacobson radical of $\Lambda$ (resp.\ $\tilde{\Lambda}_A$).
We apply this in the case where $\Lambda$ is a $d$-Koszul algebra for some $d \geqslant 2$.
This connects with work of Leader \cite{L} in which she considered a family of algebras which are
seen to be stretched algebras in the special case where $\Lambda$ is a $d$-Koszul algebra. Our approach is very different, but as a consequence and in the case where $\Lambda$ is $d$-Koszul, we recover \cite[Theorem 8.15]{L} by showing that $\tilde{\Lambda}_A$ is a $(D,A)$-stacked algebra where $D = dA$; this is Theorem~\ref{(D,A)}.
The class of $(D,A)$-stacked algebras was introduced by Leader and Snashall in \cite[Definition 2.1]{LS} (see Definition~\ref{defin:(D,A)-stacked}) and provides a natural generalisation of Koszul and $d$-Koszul algebras.
Thus Theorem~\ref{(D,A)} gives us examples of stretched algebras as well as a construction of $(D,A)$-stacked algebras.

\medskip

We keep the following notation throughout the paper. The set of vertices of a quiver $\cQ$ is denoted by $\cQ_0$.
An arrow $\alpha$ starts at
$\mo(\alpha)$ and ends at $\mt(\alpha)$; arrows in a path are read from left to right.
A path $p = \alpha_1\alpha_2 \cdots \alpha_n$, where  $\alpha_1, \alpha_2 , \dots , \alpha_n$
are arrows, is of length $n$ with $\mo(p) = \mo(\alpha_{1})$ and $\mt(p) = \mt(\alpha_{n})$.
We write $\ell(p)$ for the length of the path $p$.
An element $x$ in $K\cQ$ is {\it uniform} if there exist vertices $v, v'$ in $\cQ$
such that $x = v x = x v'$. We then write $\mo(x) = v$ and $\mt(x) = v'$.
If the ideal $I$ is generated by paths in $K\cQ$ then $K\cQ / I$ is a {\it monomial algebra}.
If $I$ is length homogeneous, then $\Lambda = \Lambda_0 \oplus \Lambda_1 \oplus \cdots$ is a
graded algebra with the length grading, and $\Lambda_0 \cong \Lambda/\rrad$.
The Ext algebra of $\Lambda$
is given by $E(\Lambda) = \oplus_{n\geqslant 0}\Ext^n_\Lambda(\Lambda/\rrad,\Lambda/\rrad)$
with the Yoneda product. The Hochschild cohomology ring of $\Lambda$ is given by
$\HH^*(\Lambda) = \Ext^*_{\Lambda^e}(\Lambda,\Lambda) = \oplus_{n\geqslant 0}\Ext^n_{\Lambda^e}(\Lambda,\Lambda)$ with the Yoneda product, where $\Lambda^e = \Lambda^\op \otimes_K \Lambda$ is the enveloping algebra of $\Lambda$.
All modules are finite-dimensional right modules.
We write $\dim$ for $\dim_K$ and $\otimes$ for $\otimes_K$; in all other cases the subscripts are specified. We use $\pdim$ for the projective dimension, $\idim$ for the injective dimension and $\gldim$ for the global dimension.

\bigskip

\section{Constructing the stretched algebra}\label{sec:dtoD}

Let $\Lambda = K\cQ/I$ be a finite-dimensional algebra where $I$ is generated by
a minimal set $g^2$ of uniform elements in $K\cQ$.
Let $A \geqslant 1$.
We describe the construction of $\tilde{\Lambda}_A$ by using the quiver $\cQ$ and ideal
$I$ of $K\cQ$ to define a new quiver $\tilde{\cQ}_A$ and admissible ideal $\tilde{I}_A$ of $K\tilde{\cQ}_A$
giving $\tilde{\Lambda}_A = K\tilde{\cQ}_A/\tilde{I}_A$. This construction builds on ideas in \cite{L}.
We begin with the quiver $\tilde{\cQ}_A$.

\begin{defin}\label{defin:quiverconstruct}
Let $\cQ$ be a finite quiver. Let $A \geqslant 1$.
We construct the new quiver $\tilde{\cQ}_A$ as follows:
\begin{enumerate}
\item[$\bullet$] All vertices of $\cQ$ are also vertices in $\tilde{\cQ}_A$.
\item[$\bullet$] For each arrow $\alpha$ in $\cQ$ we have $A$ arrows $\alpha_1, \alpha_2 , \ldots , \alpha_A$ in $\tilde{\cQ}_A$ and additional vertices $w_1 , w_2 , \ldots , w_{A-1}$ in $\tilde{\cQ}_A$, such that:\\
$\hspace*{2cm}\begin{array}{lclcl}
\mo(\alpha_1) & = & \mo(\alpha) & & \\
\mt(\alpha_1) & = & \mo(\alpha_2) & = & w_1  \\
\mt(\alpha_2) & = & \mo(\alpha_3) & = & w_2 \\
 & \vdots & & \vdots & \\
\mt(\alpha_{A-1}) & = & \mo(\alpha_A) & = & w_{A-1} \\
\mt(\alpha_A) & = & \mt(\alpha) & &
\end{array}$\\
and the only arrows incident with the vertex $w_j$ are $\alpha_j$ and $\alpha_{j+1}$.
\end{enumerate}
\end{defin}

In this way the arrow $\alpha$ in $\cQ$ corresponds to a path $\alpha_1 \cdots \alpha_A$ of length $A$ in $\tilde{\cQ}_A$. For ease of notation, we identify the set of vertices $\cQ_0$ of $\cQ$ with the corresponding subset of the vertices of $\tilde{\cQ}_A$.

\begin{defin}\label{defin:theta}
Let $\theta^*:K\cQ \rightarrow K\tilde{\cQ}_A$ be the $K$-algebra homomorphism which is induced from
$$\begin{cases}
v\mapsto v & \mbox{for each vertex $v$ in $\mathcal{Q}$},\\
\alpha \mapsto \alpha_1\alpha_2\cdots\alpha_A & \mbox{for each arrow $\alpha$ in $\mathcal{Q}$}.
\end{cases}$$
Moreover, $\theta^*$ is also a $K$-algebra monomorphism.
\end{defin}

\begin{defin}\label{shortestpath}
Suppose $w \in ({\tilde{\mathcal Q}}_A)_0\setminus\mathcal{Q}_0$.
Define $\tilde{p}_w$ to be the unique shortest path in $K\tilde{\mathcal{Q}}_A$ which starts
at a vertex in ${\mathcal{Q}}_0$ and ends at $w$. Define $\tilde{q}_w$ to be the unique
shortest path in $K\tilde{\mathcal{Q}}_A$ which starts at the vertex $w$ and ends at a vertex
in $\mathcal{Q}_0$.
\end{defin}

\begin{remark}\label{rem:w}
Let $w \in ({\tilde{\mathcal Q}}_A)_0\setminus\mathcal{Q}_0$.
Then there is a unique arrow $\alpha$ in $\cQ$ such that
$\theta^*(\alpha) = \alpha_1 \cdots \alpha_A$ and
$w = w_i$ for some $i = 1, \dots, A-1$.
Let $v = \mo(\alpha)$ and let $v' = \mt(\alpha)$.
Then the quiver $\tilde{\mathcal{Q}}_A$ contains the subquiver
$$\xymatrix{
v\ar[r]^{\alpha_1} & w_1 \ar[r]^{\alpha_2} & w_2\ar[r]^{\alpha_3} & \cdots \ar[r]^{\alpha_{A-1}} &
w_{A-1}\ar[r]^{\alpha_A} & v'
}$$
Thus $\tilde{p}_{w_i} = \alpha_1\cdots\alpha_i$ and $\tilde{q}_{w_i} = \alpha_{i+1}\cdots\alpha_A$.
Moreover $\mo(\tilde{p}_{w_i}) = v, \mt(\tilde{q}_{w_i}) = v'$ and $\tilde{p}_{w_i} \tilde{q}_{w_i} = \alpha_1 \cdots \alpha_A$.

We may illustrate these paths by:
$$\xymatrix{
v\ar@{~>}[rr]^{\tilde{p}_w} & & w\ar@{~>}[rr]^{\tilde{q}_w} & & v'
}$$
\end{remark}

We are now ready to define the algebra $\tilde{\Lambda}_A$.

\begin{defin}\label{defin:algebraconstruct}
Let $\Lambda = K\cQ/I$ be a finite-dimensional algebra where $I$ is generated by a minimal set
$g^2$ of uniform elements in $K\cQ$. List the elements of $g^2$ as $g^2_1, g^2_2,\dots, g^2_m$.
Let $A \geqslant 1$. Let $\tilde{\cQ}_A$ be the quiver defined in Definition~\ref{defin:quiverconstruct}.
For $i = 1, \dots , m$, define $\tilde{g}^2_i = \theta^*(g^2_i)$.
Then each $\tilde{g}^2_i$ is a uniform element in $K\tilde{\mathcal{Q}}_A$ with
both $\mo(\tilde{g}^2_i)$ and $\mt(\tilde{g}^2_i)$ being vertices in ${\mathcal Q}$,
namely, $\mo(\tilde{g}^2_i)=\mo(g^2_i)$ and $\mt(\tilde{g}^2_i) = \mt(g^2_i)$.
We define $\tilde{I}_A$ to be the ideal of $K\tilde{\mathcal{Q}}_A$ generated by
$\tilde{g}^2=\{\tilde{g}^2_1,\dots,\tilde{g}^2_m\}$ and define $\tilde{\Lambda}_A=K\tilde{\mathcal{Q}}_A/\tilde{I}_A$. We call $\tilde{\Lambda}_A$ the stretched algebra of $\Lambda$.
\end{defin}

\begin{example}
Let $\Lambda = K\cQ/I$ where $\cQ$ is the quiver
$$\xymatrix{
v\ar@(ul,dl)[]_{x}\ar@(ur,dr)[]^{y}
}$$
and $I=\langle x^2, xy-yx, y^2\rangle$.
Then, for $A=2$, the stretched algebra $\tilde{\Lambda}_2$ is given by $\tilde{\Lambda}_2 = K\tilde{\cQ}/\tilde{I}$ where $\tilde{\cQ}$ is the quiver
$$\xymatrix{
w\ar@<-0.5ex>@/_/[r]_{x_2} &
v\ar@<-0.5ex>@/_/[l]_{x_1}\ar@<0.5ex>@/^/[r]^{y_1} &
w'\ar@<0.5ex>@/^/[l]^{y_2}
}$$
and $\tilde{I} = \langle (x_1x_2)^2, x_1x_2y_1y_2 - y_1y_2x_1x_2, (y_1y_2)^2\rangle$.
\end{example}

This construction has the following properties.

\begin{prop}
Let $m_0$ be the number of vertices of $\mathcal{Q}$ and $m_1$  be the number of arrows of $\mathcal{Q}$. We have the following properties.
\begin{enumerate}
 \item The stretched algebra $\tilde{\Lambda}_A$ is a finite-dimensional algebra.
 \item The quiver $\tilde{\mathcal{Q}}_A$ has $m_0+m_1(A-1)$ vertices and $m_1A$ arrows.
 \item The set $\tilde{g}^2=\{\tilde{g}^2_1,\dots,\tilde{g}^2_m\}$ is a minimal generating set of uniform elements for $\tilde{I}_A$.
 \item If $I$ is generated by length homogeneous elements, then $\tilde{I}_A$ is generated by length homogeneous elements.
 \item If $I$ is generated by length homogeneous elements all of length $d$, then $\tilde{I}_A$ is generated by length homogeneous elements all of length $dA$.
 \item If $\Lambda$ is a monomial algebra, then $\tilde{\Lambda}_A$  is a monomial algebra.
\end{enumerate}
\end{prop}

To avoid too many subscripts and where there is no confusion, we write $\tilde{\Lambda}$ (resp.\ $\tilde{\mathcal{Q}}$, $\tilde{I}$) instead of $\tilde{\Lambda}_A$ (resp.\ $\tilde{\mathcal{Q}}_A$, $\tilde{I}_A$).

\begin{defin}
Let $\varepsilon = \sum_{v \in\mathcal{Q}_0} v$, which is considered as an element of $\tilde{\Lambda}$.
\end{defin}

From Definition~\ref{defin:theta} and using the construction of $\tilde{I}$, we let $\theta$ denote the map $\Lambda \to \tilde{\Lambda}$ which is induced by $\theta^*$. Note that $\theta$ is also a $K$-algebra monomorphism. It is immediate that $\varepsilon$ is an idempotent element of $\tilde{\Lambda}$ and that $\Im\theta \subseteq \varepsilon\tilde{\Lambda}\varepsilon$.
By the construction of $\tilde{\Lambda}$, if a uniform element $\tilde{p}\in\tilde{\Lambda}$ has
$\mo(\tilde{p})\in\mathcal{Q}_0$ and $\mt(\tilde{p})\in\mathcal{Q}_0$ then $\tilde{p}=\theta(p)$ for some
$p\in \Lambda$. Hence the algebras $\Lambda$ and $\varepsilon\tilde{\Lambda}\varepsilon$ are isomorphic.

\begin{thm}\label{thm:varepsilon}
Let $\Lambda=K\mathcal{Q}/I$ be a finite-dimensional algebra.
Then $\Lambda\cong \varepsilon\tilde{\Lambda}\varepsilon$.
\end{thm}

Furthermore, if $w\in\tilde{\mathcal{Q}}_0\setminus\mathcal{Q}_0$, then we observe from Remark~\ref{rem:w},
that any element of $\varepsilon K\tilde{\cQ} w$ can be written as $(\varepsilon \tilde{s} \varepsilon) v \tilde{p}_w$ for some $\tilde{s} \in K\tilde{\cQ}$. Similarly, any element of $w K\tilde{\cQ}\varepsilon$ can be written as $\tilde{q}_w v'(\varepsilon \tilde{s} \varepsilon)$ for some $\tilde{s} \in K\tilde{\cQ}$.

\begin{prop}\label{prop:theta props}
Let $w \in \tilde{\mathcal{Q}}_0\setminus\mathcal{Q}_0$.
Let $v =\mo(\tilde{p}_{w})$ and $v'=\mt(\tilde{q}_{w})$.
Let $\lambda \in \Lambda$ and $\tilde{\lambda} \in \tilde{\Lambda}$.
\begin{enumerate}
\item If $0\neq\tilde{\lambda} v\in \tilde{\Lambda}v$, then $0 \neq \tilde{\lambda}\tilde{p}_w \in \tilde{\Lambda}$.
\item If $0\neq v'\tilde{\lambda}\in v'\tilde{\Lambda}$, then $0 \neq \tilde{q}_w\tilde{\lambda} \in \tilde{\Lambda}$.
\end{enumerate}
\end{prop}

\begin{proof}
(1). Suppose that $\tilde{\lambda}\tilde{p}_w =0$ in $\tilde{\Lambda}$.
By considering $\tilde{\lambda}v$ as an element of $K\tilde{\cQ}$, we have that $\tilde{\lambda}\tilde{p}_w \in \tilde{I}$. Now, $\tilde{I}$ is generated by the set $\{\tilde{g}^2_1, \dots , \tilde{g}^2_m\}$ of uniform elements in $\varepsilon K\tilde{\cQ}\varepsilon$, so write
$\tilde{\lambda}\tilde{p}_w = \sum_k \tilde{r}_k\varepsilon \tilde{g}^2_k\varepsilon \tilde{s}_k w$
for some $\tilde{r}_k, \tilde{s}_k$ in $K\tilde{\cQ}$. As noted above, each term $\varepsilon \tilde{s}_k w$ is of the form $\varepsilon \tilde{t}_k v \tilde{p}_w$ for some $\tilde{t}_k \in K\tilde{\cQ}$.
So we have $\tilde{\lambda}\tilde{p}_w = (\sum_k \tilde{r}_k\varepsilon \tilde{g}^2_k\varepsilon \tilde{t}_k v)\tilde{p}_w$ in $K\tilde{\cQ}$ and hence $\tilde{\lambda}v \in \tilde{I}$. Thus $\tilde{\lambda} v = 0$ in $\tilde{\Lambda}v$ as required.

The proof of (2) is similar.
\end{proof}

The fact that $\tilde{I}$ is generated by uniform elements in $\varepsilon K\tilde{\cQ}\varepsilon$ is used again in the proofs of the next two propositions; they are straightforward and are left to the reader.

\begin{prop}\label{prop:iso modules}
Let $w \in \tilde{\mathcal{Q}}_0 \setminus Q_0$.
Let $v =\mo(\tilde{p}_{w})$ and $v'=\mt(\tilde{q}_{w})$. Let $B=\varepsilon\tilde{\Lambda}\varepsilon$.
Then we have the following properties:
\begin{enumerate}
\item $v'B \cong \tilde{q}_wB$ as right $B$-modules.
\item $v'\tilde{\Lambda} \cong \tilde{q}_w \tilde{\Lambda}$ as right $\tilde{\Lambda}$-modules.
\item $Bv \cong B\tilde{p}_w$ as left $B$-modules.
\item $\tilde{\Lambda}v \cong \tilde{\Lambda}\tilde{p}_w$ as left $\tilde{\Lambda}$-modules.
\end{enumerate}
\end{prop}

\begin{prop}\label{prop:dimension of wiL}
Let $w \in \tilde{\mathcal{Q}}_0 \setminus Q_0$. We use the notation of Remark~\ref{rem:w}, so $w = w_i$ for some $i=1, \dots, A-1$.
\begin{enumerate}
\item An element of $\tilde{\Lambda}w_i$ is of the form
$$\tilde{\lambda}w_i=\sum_{j=1}^i c_jw_j\alpha_{j+1}\cdots\alpha_iw_i+\tilde{\mu}\tilde{p}_{w_i}$$
where $c_j\in K$, $\tilde{\mu}\in\tilde{\Lambda}$.
\item An element of $w_i\tilde{\Lambda}$ is of the form
$$w_i\tilde{\lambda} =\sum_{j=i}^{A-1} c_jw_i\alpha_{i+1}\cdots\alpha_jw_j+\tilde{q}_{w_i}\tilde{\mu}$$
where $c_j\in K$, $\tilde{\mu}\in\tilde{\Lambda}$.
\item $\dim \tilde{\Lambda}w_i = i+\dim\tilde{\Lambda}v$.
\item $\dim w_i\tilde{\Lambda} = (A-i)+\dim v'\tilde{\Lambda}$.
\item $\varepsilon \tilde{\Lambda}w = \varepsilon\tilde{\Lambda}\varepsilon\tilde{p}_w$ and  $w\tilde{\Lambda}\varepsilon = \tilde{q}_w\varepsilon\tilde{\Lambda}\varepsilon$.
\end{enumerate}
\end{prop}

\begin{thm}\label{projBmodule}
Let $\Lambda=K\mathcal{Q}/I$ and let $\tilde{\Lambda}$ be the stretched algebra.
Let $B=\varepsilon\tilde{\Lambda}\varepsilon$. Then
\begin{enumerate}
\item $\tilde{\Lambda}\varepsilon$ is projective as a right $B$-module.
\item $\varepsilon\tilde{\Lambda}$ is projective as a left $B$-module.
\end{enumerate}
\end{thm}

\begin{proof}
(1). We have that $\tilde{\Lambda}\varepsilon
= \varepsilon\tilde{\Lambda}\varepsilon \oplus (1-\varepsilon)\tilde{\Lambda}\varepsilon
= \varepsilon\tilde{\Lambda}\varepsilon \oplus \big(\oplus_{w \in \tilde{\mathcal{Q}}_0\setminus{\mathcal Q}_0}w\tilde{\Lambda}\varepsilon\big)$.
From Proposition~\ref{prop:dimension of wiL}(5) and Proposition~\ref{prop:iso modules}(1), we have that $w\tilde{\Lambda}\varepsilon = \tilde{q}_w\varepsilon\tilde{\Lambda}\varepsilon = \tilde{q}_w B \cong \mt(\tilde{q}_w)B$.
Thus, $\tilde{\Lambda}\varepsilon \cong B \oplus \big(\oplus_{w \in \tilde{\mathcal{Q}}_0\setminus\mathcal{Q}_0}\mt(\tilde{q}_w)B\big)$. Noting that each $\mt(\tilde{q}_w)$ is a vertex in $\cQ_0$, it follows that $\tilde{\Lambda}\varepsilon$ is a projective right $B$-module.

The proof of (2) is similar.
\end{proof}

\section{Stratifying ideals and the {\rm\bf (Fg)} condition}\label{sec:Fg}

We consider the finiteness condition {\bf (Fg)} under which we have a rich theory of support varieties for modules over a finite-dimensional algebra. Our first result is Theorem~\ref{thm:stratifying ideal}, which shows that the ideal $\langle\varepsilon\rangle$ is a stratifying ideal of the stretched algebra $\tilde{\Lambda}$. We start by recalling the definition of a stratifying ideal.

\begin{defin}
Let $A$ be a finite-dimensional algebra and let $e$ be an idempotent in $A$.
The two sided ideal $\langle e\rangle = AeA$ is a stratifying ideal if:
\begin{enumerate}
\item the multiplication map $Ae\otimes_{eAe}eA\rightarrow AeA$ is an isomorphism, and
\item $\Tor^{eAe}_n(Ae,eA)=0$ for all $n>0$.
\end{enumerate}
\end{defin}

It is clear that if the multiplication map $Ae\otimes_{eAe}eA\rightarrow AeA$ is an isomorphism and $Ae$ is a projective right $eAe$-module, then $\langle e\rangle$ is a stratifying ideal.

\begin{thm}\label{thm:stratifying ideal}
Let $\Lambda=K\mathcal{Q}/I$ and let $\tilde{\Lambda}$ be the stretched algebra. Recall that $\varepsilon=\sum_{v\in{\mathcal{Q}_0}}v$ and $B=\varepsilon\tilde{\Lambda}\varepsilon$.
Then $\langle\varepsilon\rangle$ is a stratifying ideal of $\tilde{\Lambda}$.
\end{thm}

\begin{proof}
From Theorem~\ref{projBmodule}, $\tilde{\Lambda}\varepsilon$ is projective as a right $B$-module. So it suffices to show that the multiplication map
$\psi: \tilde{\Lambda}\varepsilon\otimes_B\varepsilon\tilde{\Lambda}\rightarrow \tilde{\Lambda}\varepsilon\tilde{\Lambda}$
is an isomorphism. It is clear that $\psi$ is a $\tilde{\Lambda}$-$\tilde{\Lambda}$-bimodule homomorphism and is onto. We show that $\psi$ is one-to-one.

Suppose that $\psi(\sum \tilde{\lambda}\varepsilon\otimes_B\varepsilon\tilde{\mu}) = 0$, with $\tilde{\lambda}, \tilde{\mu} \in \tilde{\Lambda}$.
From Proposition~\ref{prop:dimension of wiL}(5), $\tilde{\Lambda}\varepsilon = \varepsilon\tilde{\Lambda}\varepsilon \oplus \big( \oplus_{w \in \tilde{\mathcal{Q}}_0\setminus\mathcal{Q}_0} w\tilde{\Lambda}\varepsilon) = \varepsilon\tilde{\Lambda}\varepsilon \oplus \big( \oplus_{w \in \tilde{\mathcal{Q}}_0\setminus\mathcal{Q}_0} \tilde{q}_w\varepsilon\tilde{\Lambda}\varepsilon)$ so we may write
$$\sum \tilde{\lambda}\varepsilon\otimes_B\varepsilon\tilde{\mu} = \varepsilon\otimes_B \varepsilon\tilde{\nu}+\sum_{w \in \tilde{\mathcal{Q}}_0\setminus\mathcal{Q}_0}\tilde{q}_w\varepsilon\otimes_B \varepsilon\tilde{\nu}_w$$
for some $\tilde{\nu}, \tilde{\nu}_w$ in $\tilde{\Lambda}$.
Then $0 = \psi(\sum\tilde{\lambda}\varepsilon\otimes_B\varepsilon\tilde{\mu}) =
\varepsilon\tilde{\nu} +
\sum_{w \in \tilde{\mathcal{Q}}_0\setminus\mathcal{Q}_0}\tilde{q}_w\tilde{\nu}_w$.
Left multiplication by $\varepsilon$ gives that $\varepsilon\tilde{\nu} = 0$.
For each $w \in \tilde{\mathcal{Q}}_0\setminus{\mathcal Q}_0$, left multiplication by $w$ gives $\tilde{q}_w\tilde{\nu}_w=0$; then
from Proposition~\ref{prop:theta props}, we have that $\mt(\tilde{q}_w)\tilde{\nu}_w=0$.
Thus $\sum \tilde{\lambda}\varepsilon\otimes_B\varepsilon\tilde{\mu} = 0$ and $\psi$ is one-to-one.
Hence $\langle\varepsilon\rangle$ is a stratifying ideal.
\end{proof}

We now study the quotient $\tilde{\Lambda}/\langle\varepsilon\rangle$.
We use the notation introduced in Remark~\ref{rem:w}. In addition, for each
arrow $\alpha$ in $\cQ$, let $\Gamma_\alpha$ denote the following subquiver of $\tilde{\mathcal{Q}}$
$$\xymatrix{
w_1 \ar[r]^{\alpha_2} & w_2\ar[r]^{\alpha_3} & \cdots \ar[r]^{\alpha_{A-1}} &
w_{A-1}
}$$
We have
$$\tilde{\Lambda}/\tilde{\Lambda}\varepsilon\tilde{\Lambda} \cong \oplus_{\alpha\in{\mathcal{Q}_1}}(\tilde{\Lambda}w_1\tilde{\Lambda}+\tilde{\Lambda}w_2\tilde{\Lambda}+ \cdots +\tilde{\Lambda}w_{A-1}\tilde{\Lambda} + \tilde{\Lambda}\varepsilon\tilde{\Lambda})/\tilde{\Lambda}\varepsilon\tilde{\Lambda}.$$
Define
$$X_\alpha=(\tilde{\Lambda}w_1\tilde{\Lambda}+\tilde{\Lambda}w_2\tilde{\Lambda}+ \cdots + \tilde{\Lambda}w_{A-1}\tilde{\Lambda} + \tilde{\Lambda}\varepsilon\tilde{\Lambda})/\tilde{\Lambda}\varepsilon\tilde{\Lambda}$$
so
$$ \tilde{\Lambda} / \tilde{\Lambda}\varepsilon\tilde{\Lambda} \cong \oplus_{\alpha\in{\mathcal{Q}_1}}X_\alpha.$$
Moreover, $X_\alpha \cong K\Gamma_\alpha$ as $K$-algebras.
The following result is now immediate.

\begin{prop}\label{prop:dimX}
Let $\Lambda = K\mathcal{Q}/I$ and let $\tilde{\Lambda}$ be the stretched algebra.
Then $\dim X_{\alpha} = A(A-1)/2$ and
$\dim\tilde{\Lambda} / \langle\varepsilon\rangle = m_1A(A-1)/2$, where $m_1$ is the number of arrows of $\mathcal{Q}$.
\end{prop}

\begin{thm}
Let $\Lambda=K\mathcal{Q}/I$ and let $\tilde{\Lambda}$ be the stretched algebra.
Then $\tilde{\Lambda}/\langle\varepsilon\rangle$ has a minimal projective
$\tilde{\Lambda}$-$\tilde{\Lambda}$-bimodule resolution
$$0 \rightarrow \tilde{R}^2 \rightarrow \tilde{R}^1 \rightarrow \tilde{R}^0 \rightarrow \tilde{\Lambda}/\langle\varepsilon\rangle\rightarrow 0.$$
\end{thm}

\begin{proof}
The main part of this proof is in constructing a minimal projective $\tilde{\Lambda}$-$\tilde{\Lambda}$-bimodule resolution for each algebra $X_\alpha$.

Let $\alpha$ be an arrow in $\cQ$. We keep the notation of this section and of Remark~\ref{rem:w}.
Let $v = \mo(\alpha)$ and $v' = \mt(\alpha)$. Set $\dim\tilde{\Lambda}v=V$ and $\dim v'\tilde{\Lambda}= V'$.

Define the bimodule $\tilde{R}^0_\alpha = \oplus_{i=1}^{A-1}\tilde{\Lambda}w_i\otimes w_i\tilde{\Lambda}$ and the bimodule homomorphism $\Delta^0_\alpha : \tilde{R}^0_\alpha\rightarrow X_\alpha$ by
$w_i\otimes w_i \mapsto w_i+\tilde{\Lambda}\varepsilon\tilde{\Lambda}$ for $i=1,\dots,A-1$.
Using Proposition~\ref{prop:dimension of wiL}, we have
\begin{align*}
\dim\tilde{R}^0_\alpha & = \sum_{i=1}^{A-1}\dim(\tilde{\Lambda}w_i)\dim(w_i\tilde{\Lambda})\\
&= \sum_{i=1}^{A-1}(i+V)((A-i)+V')\\
&= \sum_{i=1}^{A-1}i(A-i) +\sum_{i=1}^{A-1}i(V+V') + (A-1)VV'\\
&= \frac{1}{6}(A-1)A(A+1) + \frac{1}{2}(A-1)A(V+V') + (A-1)VV'.
\end{align*}
So, with Proposition~\ref{prop:dimX}, we have
\begin{align*}
\dim\Ker\Delta^0_\alpha & = \dim\tilde{R}^0_\alpha - \dim X_\alpha\\
&= \frac{1}{6}(A-2)(A-1)A + \frac{1}{2}(A-1)A(V+V') + (A-1)VV'.
\end{align*}

The next step is to find the generators of $\Ker \Delta^0_\alpha$.
Let $K$ be the $\tilde{\Lambda}$-$\tilde{\Lambda}$-bimodule generated by $\{\tilde{p}_{w_1}\otimes w_1, w_{A-1}\otimes\tilde{q}_{w_{A-1}}, w_{i}\otimes\alpha_{i+1}-\alpha_{i+1}\otimes w_{i+1}, \mbox{ for } i=1,\dots,A-2\}$. Clearly $K \subseteq \Ker\Delta^0_\alpha$.
For the reverse inclusion, suppose first that $A=2$.
Set $U_1 = \tilde{\Lambda}\tilde{p}_{w_1} \otimes w_1\tilde{\Lambda}$ and $U_2 = \tilde{\Lambda}w_1 \otimes \tilde{q}_{w_1}\tilde{\Lambda}$, and note that both $\tilde{p}_{w_1}$ and $\tilde{q}_{w_1}$ are arrows in $\tilde{{\mathcal Q}}$. Then $K = U_1 + U_2$. So $\dim K = \dim U_1 + \dim U_2 - \dim (U_1 \cap U_2)$. We see that $U_1 \cap U_2 = \tilde{\Lambda}\tilde{p}_{w_1} \otimes \tilde{q}_{w_1}\tilde{\Lambda}$. So, from Propositions~\ref{prop:iso modules} and \ref{prop:dimension of wiL},
$$\dim K = V(1+V') + (1+V)V' - VV' = V+V'+VV'.$$
So $\dim K = \dim\Ker\Delta^0_{\alpha}$ and thus $K = \Ker\Delta^0_{\alpha}$.

Now suppose that $A \geqslant 3$.
Here we set $U_1=\tilde{\Lambda} \tilde{p}_{w_1}\otimes\ w_1 \tilde{\Lambda}$,
$U_2=\tilde{\Lambda} w_{A-1}\otimes \tilde{q}_{w_{A-1}}\tilde{\Lambda}$
and $U_3=\sum_{i=1}^{A-2}\tilde{\Lambda}( w_{i}\otimes\alpha_{i+1}- \alpha_{i+1}\otimes w_{i+1})\tilde{\Lambda}$. Then $K = U_1 + U_2 + U_3$ so $\dim K = \dim (U_1+U_2) + \dim U_3 - \dim(U_1 + U_2) \cap U_3$. Now, $U_1 \subseteq \tilde{\Lambda}w_1 \otimes w_1\tilde{\Lambda}$ and $U_2 \subseteq \tilde{\Lambda}w_{A-1} \otimes w_{A-1}\tilde{\Lambda}$, so, since $A \geqslant 3$, we have $U_1 \cap U_2 = \{0\}$.
Thus $\dim (U_1+U_2)=\dim U_1+\dim U_2$. Note also that $\tilde{\Lambda}(w_{i}\otimes\alpha_{i+1}- \alpha_{i+1}\otimes w_{i+1})\tilde{\Lambda} \cong \tilde{\Lambda}(w_{i}\otimes w_{i+1})\tilde{\Lambda}$
so $U_3 \cong \oplus_{i=1}^{A-2}\tilde{\Lambda}(w_{i}\otimes w_{i+1})\tilde{\Lambda}$. Then Propositions~\ref{prop:iso modules} and \ref{prop:dimension of wiL} give
\begin{align*}
\dim U_1&= V((A-1)+V')\\
\dim U_2&= ((A-1)+V)V'\\
\dim U_3&= \sum_{i=1}^{A-2}(i+V)((A-(i+1))+V').
\end{align*}
Finally, we can write
$$\tilde{p}_{w_1}\otimes\tilde{q}_{w_1} - \tilde{p}_{w_{A-1}} \otimes \tilde{q}_{w_{A-1}}
= \sum_{j=1}^{A-2} \tilde{p}_{j}(w_j\otimes\alpha_{j+1} - \alpha_{j+1}\otimes w_{j+1})\tilde{q}_{j+1}$$
so that $\tilde{p}_{w_1} \otimes \tilde{q}_{w_1} - \tilde{p}_{w_{A-1}} \otimes \tilde{q}_{w_{A-1}} \in (U_1 + U_2) \cap U_3$. Indeed, this element generates $(U_1 + U_2) \cap U_3$ so $\dim((U_1 + U_2) \cap U_3)= VV'$.
Hence
\begin{align*}
\dim K &=
V((A-1)+V')+((A-1)+V)V'+\left(\sum_{i=1}^{A-2}(i+V)((A-(i+1))+V')\right)-VV'\\
&=(A-1)(V+V')+VV'+\sum_{i=1}^{A-2}i(A-i-1) + \sum_{i=1}^{A-2}i(V+V') + \sum_{i=1}^{A-2}VV' \\
&=\sum_{i=1}^{A-2}i(A-i-1) + \sum_{i=1}^{A-1}i(V+V') + \sum_{i=1}^{A-1}VV' \\
&= \frac{1}{6}(A-2)(A-1)A + \frac{1}{2}(A-1)A(V+V') + (A-1)VV'\\
&=\dim \Ker\Delta^0_\alpha
\end{align*}
and so $K = \Ker\Delta^0_\alpha$.

Next we define the bimodule $\tilde{R}^1_\alpha = \tilde{\Lambda}v\otimes w_1\tilde{\Lambda}
\oplus (\oplus_{i=1}^{A-2}\tilde{\Lambda}w_i\otimes w_{i+1}\tilde{\Lambda}) \oplus \tilde{\Lambda}w_{A-1}\otimes v'\tilde{\Lambda}$ and the bimodule homomorphism
$\Delta^1_\alpha:\tilde{R}^1_\alpha\rightarrow\tilde{R}^0_\alpha$ by
$$\begin{cases}
v\otimes w_1\mapsto \tilde{p}_{w_1}\otimes w_1\\
w_i\otimes {w_{i+1}}\mapsto w_{i}\otimes\alpha_{i+1}-\alpha_{i+1}\otimes w_{i+1}\\
w_{A-1}\otimes v'\mapsto {w_{A-1}}\otimes \tilde{q}_{w_{A-1}}
\end{cases}$$
where $\tilde{p}_{w_1}\otimes w_1$ lies in the $w_1\otimes w_1$-component of $\tilde{R}^0_\alpha$,
$w_{A-1}\otimes \tilde{q}_{w_{A-1}}$ lies in the $w_{A-1}\otimes w_{A-1}$-component of $\tilde{R}^0_\alpha$,
and, for $i = 1, \dots , A-2$,
$w_i\otimes\alpha_{i+1}$ lies in the $w_i\otimes w_i$-component of $\tilde{R}^0_\alpha$, and
$\alpha_{i+1}\otimes w_{i+1}$ lies in the $w_{i+1}\otimes w_{i+1}$-component of $\tilde{R}^0_\alpha$.
Then
\begin{align*}
\dim \tilde{R}^1_\alpha
&= V((A-1)+V')+ \sum_{i=1}^{A-2}(i+V)((A-(i+1))+V')+((A-1)+V)V'\\
&= \sum_{i=1}^{A-1}i(V+V') + AVV'+\sum_{i=1}^{A-2}i(A-(i+1))\\
&= \frac{1}{2}A(A-1)(V+V') + AVV' + \frac{1}{6}A(A-1)(A-2)
\end{align*}
and hence $\dim\Ker\Delta^1_\alpha=\dim \tilde{R}^1_\alpha-\dim\Ker\Delta^0_\alpha=VV'$.
To find $\Ker \Delta^1_\alpha$, let
$$z = (v\otimes \tilde{q}_{w_1},-\tilde{p}_{w_1}\otimes \tilde{q}_{w_2}, \dots,-\tilde{p}_{w_i}\otimes \tilde{q}_{w_{i+1}},\dots, -\tilde{p}_{w_{A-2}}\otimes \tilde{q}_{w_{A-1}},-\tilde{p}_{w_{A-1}}\otimes v').$$
Then $z$ is in $\Ker\Delta^1_\alpha$ and generates a sub-bimodule of $\Ker\Delta^1_\alpha$ of dimension $VV'$. Hence $\Ker\Delta^1_\alpha = \langle z\rangle$.

Now define the bimodule $\tilde{R}^2_\alpha=\tilde{\Lambda}v\otimes v'\tilde{\Lambda}$ and the bimodule homomorphism $\Delta^2_\alpha : \tilde{R}^2_\alpha\rightarrow\tilde{R}^1_\alpha$ by
$v\otimes v'\mapsto z$. Then $\dim \tilde{R}^2_\alpha=VV'$ and so $\dim \Ker \Delta^2_\alpha=0$.
Thus, $X_\alpha$ has minimal projective $\tilde{\Lambda}$-$\tilde{\Lambda}$-bimodule resolution
$$0\longrightarrow \tilde{R}^2_\alpha\stackrel{\Delta^2_\alpha}{\longrightarrow} \tilde{R}^1_\alpha\stackrel{\Delta^1_\alpha}{\longrightarrow} \tilde{R}^0_\alpha\stackrel{\Delta^0_\alpha}{\longrightarrow} X_\alpha\longrightarrow 0.$$
The result now follows.
\end{proof}

\begin{cor}\label{cor:pdim}
Let $\Lambda = K\mathcal{Q}/I$ and let $\tilde{\Lambda}$ be the stretched algebra. Then $\pdim_{\tilde{\Lambda}^e} \tilde{\Lambda}/\langle\varepsilon\rangle=2$.
\end{cor}

We are now in a position to compare the Hochschild cohomology rings of $\Lambda$ and the stretched algebra $\tilde{\Lambda}$. We assume for the remainder of this section that $K$ is an algebraically closed field, and recall, for a finite-dimensional $K$-algebra $A$, that we have the natural ring homomorphism $A/\rrad \otimes_A -: \HH^*(A)\rightarrow E(A)$.

\begin{defin}\cite{EHSST}\label{defin:fg}
Let $A$ be an indecomposable finite-dimensional algebra over an algebraically closed field $K$. Then $A$ has
{\bf (Fg)} if $A$ satisfies the following two conditions:
\begin{enumerate}
\item[{\bf (Fg1)}] There is a commutative Noetherian graded subalgebra $H$ of $\HH^*(A)$ such that $H^0 = \HH^0(A)$.
\item[{\bf (Fg2)}] $E(A)$ is a finitely generated $H$-module.
\end{enumerate}
\end{defin}

As a consequence, if $A$ has {\bf (Fg)} then both $\HH^*(A)$ and $E(A)$ are finitely generated as
$K$-algebras. Moreover, it was shown in \cite[Proposition 2.5(a)]{EHSST}, that if $A$ has {\bf (Fg)} then $A$ is Gorenstein. In \cite{N}, Nagase studied the finiteness condition {\bf (Fg)} for Nakayama algebras,
proving in \cite[Corollary 10]{N} that a Nakayama algebra is Gorenstein if and only if it satisfies {\bf (Fg)}. Stratifying ideals played a key role in this work; we use the following result from \cite{N}.

\begin{prop}\cite[Proposition 6]{N}
Let $A$ be a finite-dimensional algebra over an algebraically closed field $K$ with a stratifying ideal $\langle e\rangle$.
Suppose $\pdim_{A^e}A/\langle e\rangle < \infty$. Then we have:
\begin{enumerate}
\item $\HH^{\geqslant n}(A)\cong\HH^{\geqslant n}(eAe)$ as graded algebras, where $n=\pdim_{A^e} A/\langle e\rangle + 1$,
\item $A$ satisfies {\bf (Fg)} if and only if $eAe$ satisfies {\bf (Fg)},
\item $A$ is Gorenstein if and only if $eAe$ is Gorenstein.
\end{enumerate}
\end{prop}

Combining this with Corollary~\ref{cor:pdim} gives the following result for stretched algebras.

\begin{thm}\label{thm:fg}
Let $K$ be an algebraically closed field. Let $\Lambda = K\mathcal{Q}/I$ and let $\tilde{\Lambda}$ be the stretched algebra, so that $\langle\varepsilon\rangle$ is a stratifying ideal of $\tilde{\Lambda}$.
Then:
\begin{enumerate}
\item $\HH^{\geqslant 3}({\Lambda})\cong \HH^{\geqslant 3}(\tilde{\Lambda})$ as graded algebras.
\item $\tilde{\Lambda}$ satisfies {\bf (Fg)} if and only if $\Lambda$ satisfies {\bf (Fg)}.
\item $\tilde{\Lambda}$ is Gorenstein if and only if $\Lambda$ is Gorenstein.
\end{enumerate}
\end{thm}

More recently, Psaroudakis, Skarts{\ae}terhagen and Solberg \cite{PSS} considered this finiteness condition for recollements of abelian categories, introducing the concept of an eventually homological isomorphism.
In particular, for a finite-dimensional algebra $A$ with an idempotent $e$ over an algebraically closed field $K$, they
determine when the functor $\res_e: \mod A \to \mod eAe$ in a recollement of abelian categories
is an eventually homological isomorphism.

\begin{defin}\cite[Section 3]{PSS}
Given a functor $F:\mathcal{B}\rightarrow\mathcal{C}$ between abelian categories and an integer $t$, the functor $F$ is called a $t$-homological isomorphism if there is a group isomorphism $$\Ext^j_{\mathcal{B}}(B,B')\cong\Ext^j_{\mathcal{C}}(F(B),F(B'))$$
for every pair of objects $B,B'$ in $\mathcal{B}$, and every $j>t$. Note that we do not require these isomorphisms to be induced by the functor $F$. If $F$ is a $t$-homological isomorphism for some $t$, then we say that $F$ is an eventually homological isomorphism.
\end{defin}

\begin{prop}\cite[Lemma 8.23(ii) and proof]{PSS}\label{prop:pss}
Let $A$ be a finite-dimensional algebra over an algebraically closed field $K$. Suppose that $\langle e\rangle$ is a stratifying ideal in $A$. Then the following are equivalent:
\begin{enumerate}
\item $\pdim_{A^e}A/\langle e\rangle < \infty$.
\item The functor $\res_e:\mod A\rightarrow\mod eAe$ is an eventually homological isomorphism.
\end{enumerate}
Moreover, if $\pdim_{A^e}A/\langle e\rangle = t < \infty$ then the functor $\res_e$ is a $t$-homological isomorphism.
\end{prop}

We come to the final result of this section.

\begin{thm}\label{injectivedim}
Let $K$ be an algebraically closed field. Let $\Lambda=K\mathcal{Q}/I$ and let $\tilde{\Lambda}$ be the stretched algebra, so that $\langle\varepsilon\rangle$ is a stratifying ideal of $\tilde{\Lambda}$.
Then the functor $\res_\varepsilon:\mod \tilde{\Lambda}\rightarrow\mod \varepsilon\tilde{\Lambda}\varepsilon$ is a 2-homological isomorphism and hence an eventually homological isomorphism.
Moreover, $\idim_{\tilde{\Lambda}}\tilde{\Lambda}\leqslant \sup\{\idim_\Lambda \Lambda, 2\}$.
\end{thm}

\begin{proof}
From Corollary~\ref{cor:pdim} and Proposition~\ref{prop:pss}, the functor $\res_\varepsilon:\mod \tilde{\Lambda}\rightarrow\mod \varepsilon\tilde{\Lambda}\varepsilon$ is a 2-homological isomorphism.

The inequality certainly holds if $\Lambda$ has infinite injective dimension, so assume $\idim_\Lambda \Lambda = n < \infty$ and let $m=\max \{\idim_\Lambda \Lambda, 2\}+1$.
Then
$$\Ext^m_{\tilde{\Lambda}}(X,Y)\cong\Ext^m_{\varepsilon\tilde{\Lambda}\varepsilon}(\res_\varepsilon(X),\res_\varepsilon(Y))$$
for all $X,Y\in\mod\tilde{\Lambda}$. Setting $Y=\tilde{\Lambda}$ gives
$$\Ext^m_{\tilde{\Lambda}}(X,\tilde{\Lambda}) \cong \Ext^m_{\varepsilon\tilde{\Lambda}\varepsilon}(\res_\varepsilon(X),\res_\varepsilon(\tilde{\Lambda})) \cong \Ext^m_{\varepsilon\tilde{\Lambda}\varepsilon}(\res_\varepsilon (X),\tilde{\Lambda}\varepsilon).$$
From Theorem~\ref{projBmodule}(1), $\tilde{\Lambda}\varepsilon$ is projective as a right $\varepsilon\tilde{\Lambda}\varepsilon$-module, so $\idim_{\varepsilon\tilde{\Lambda}\varepsilon}\tilde{\Lambda}\varepsilon\leqslant n$ and thus  $\Ext^{n+1}_{\varepsilon\tilde{\Lambda}\varepsilon}(\res_\varepsilon(X),\tilde{\Lambda}\varepsilon)=0$. Hence $\Ext^m_{\tilde{\Lambda}}(X,\tilde{\Lambda})=0$ and $\idim_{\tilde{\Lambda}}\tilde{\Lambda}\leqslant m-1=\max\{\idim_{\Lambda}\Lambda,2\}$ as required.
\end{proof}

\begin{example}\label{example:fg}
\begin{enumerate}
\item  Let $\Lambda = K[x]/\langle x^n\rangle$ for some $n\geqslant 2$. Let $A \geqslant 2$.
Then the stretched algebra $\tilde{\Lambda}$ has quiver
$$\xymatrix{
1 \ar[r]^{\alpha_1} & 2\ar[r]^{\alpha_2} & \cdots \ar[r]^{\alpha_{A-1}} &
A\ar@<1ex>@/^/@/^1pc/[lll]^{\alpha_A}
}$$
and $\tilde{I} = \langle (\alpha_1\cdots \alpha_A)^n\rangle$.
This is the algebra of \cite[Example 8.14]{PSS} with $m=A$.
\item Let $\Lambda=K\mathcal{Q}/I$ where $\cQ$ is the quiver
$$\xymatrix{
1\ar@<0.5ex>@/^/[r]^{\alpha} & 2 \ar@<0.5ex>@/^/[l]^{\beta}
}$$
and $I=\langle\alpha\beta\alpha, \beta\alpha\beta\rangle$.
Let $A=2$. Then the stretched algebra $\tilde{\Lambda}$ has quiver
$$\xymatrix{
\bullet\ar[r]^{\alpha_1} & \bullet\ar[d]^{\alpha_2}\\
\bullet\ar[u]^{\beta_2} & \bullet\ar[l]^{\beta_1}
}$$
and $\tilde{I}=\langle\alpha_1\alpha_2\beta_1\beta_2\alpha_1\alpha_2, \beta_1\beta_2\alpha_1\alpha_2\beta_1\beta_2\rangle$.
The stretched algebra $\tilde{\Lambda}$ is the algebra of \cite[Example 3.2]{FS}, where it was shown that $\tilde{\Lambda}$ has {\bf (Fg)} and that $\idim_{\tilde{\Lambda}}\tilde{\Lambda} = 2$. Thus we may use Theorem~\ref{thm:fg} to show that $\Lambda$ has {\bf (Fg)}. Moreover, it is immediate that $\Lambda$ is self-injective, so that the upper bound on $\idim_{\tilde{\Lambda}}\tilde{\Lambda}$ in Theorem~\ref{injectivedim} is achieved.
\end{enumerate}
\end{example}

\section{Minimal projective resolutions and $d$-Koszul algebras}\label{sec:minprojres}

In this section we keep the original assumptions, so that $K$ is a field, but is not necessarily algebraically closed, $\Lambda = K\cQ/I$ is a finite-dimensional algebra, $A \geqslant 1$, and $\tilde{\Lambda}$ is the stretched algebra. Then $I$ is generated by a minimal set $g^2$ of uniform elements in $K\cQ$, and $\tilde{I}$ is generated by the minimal set $\tilde{g}^2$ of uniform elements in $K\tilde{\cQ}$.

With the notation of \cite[Chapter I.6]{ASS}, in addition to the functor $\res_\varepsilon$ used above, we also have functors $T_\varepsilon, L_\varepsilon: \mod \varepsilon\tilde{\Lambda}\varepsilon \to \mod \tilde{\Lambda}$
so that $(T_\varepsilon, \res_\varepsilon, L_\varepsilon)$ is an adjoint triple connecting $\mod\varepsilon\tilde{\Lambda}\varepsilon$ and $\mod\tilde{\Lambda}$, namely:
$$\xymatrix{
\mod \tilde{\Lambda}\ar[rr]^{\res_\varepsilon} & & \mod \varepsilon\tilde{\Lambda}\varepsilon \ar@<-2ex>@/_1.2pc/[ll]_{T_\varepsilon}\ar@<2ex>@/^1.2pc/[ll]^{L_\varepsilon}
}$$
with $\res_\varepsilon(-) = (-)\varepsilon,\
T_\varepsilon(-) = -\otimes_{\varepsilon\tilde{\Lambda}\varepsilon}\varepsilon \tilde{\Lambda}$
and $L_\varepsilon(-) = \Hom_{\varepsilon\tilde{\Lambda}\varepsilon}(\tilde{\Lambda}\varepsilon, -)$.
The functor $T_\varepsilon$ carries projectives to projectives, and is an exact functor by Proposition~\ref{projBmodule}(2).
Using Theorem~\ref{thm:varepsilon}, we identify $\Lambda$ with $\varepsilon\tilde{\Lambda}\varepsilon$, and $\rrad$ with $\varepsilon\tilde{\rrad}\varepsilon$.

The main result of this section is Theorem~\ref{thm:projres} which takes a minimal projective resolution $(P^n, d^n)$ of $\Lambda/\rrad$ as a right $\Lambda$-module as given by Green, Solberg and Zacharia in \cite{GSZ}, and uses it to construct a minimal projective resolution $(\tilde{P}^n, \tilde{d}^n)$ of $\tilde{\Lambda}/ \tilde{\rrad}$ as a right $\tilde{\Lambda}$-module. We end the paper with an application to $d$-Koszul algebras.

We recall briefly the construction of \cite{GSZ}. Let $g^0$ be the set of vertices of $\cQ$, $g^1$ the set of arrows of $\cQ$, and $g^2$ the minimal generating for $I$ as above.
In \cite{GSZ}, the authors show that there are sets $g^n$ of uniform elements in $K\cQ$, for $n \geqslant 3$, such that for each $x \in g^n$ we have $x = \sum_i g^{n-1}_i r_i = \sum_j g^{n-2}_j s_j$ for unique $r_i \in K\cQ, s_j \in I$.
The sets $g^n$ can be chosen so that $(P^n, d^n)$ is a minimal projective resolution of $\Lambda/\rrad$ with the following properties:
\begin{enumerate}
\item[$\bullet$] $P^n = \oplus_{i} \mt(g^n_i) \Lambda$, for all $n \geqslant 0$;
\item[$\bullet$] $d^0 : P^0 \rightarrow \Lambda / \rrad$ is the canonical surjection;
\item[$\bullet$] for each $n \geqslant 1$ and $x \in g^n$ there are unique elements $r_i \in K\cQ$ with $x = \sum_i g^{n-1}_i r_i$;
\item[$\bullet$] for each $n \geqslant 1$ and for $x \in g^n$, the $\Lambda$-homomorphism $d^n : P^n \rightarrow P^{n-1}$ is such that $d^n(\mt(x))$ has entry $\mt(g^{n-1}_i)r_i$ in the summand of $P^{n-1}$ corresponding to $\mt(g^{n-1}_i)$.
\end{enumerate}

We come now to Theorem~\ref{thm:projres}. Note that the proof requires a technical result which we state and prove in Proposition~\ref{prop for thm} immediately following the theorem.

\begin{thm}\label{thm:projres}
Let $\Lambda=K\mathcal{Q}/I$ and let $\tilde{\Lambda}$ be the stretched algebra.
Let $(P^n, d^n)$ be a minimal projective resolution for $\Lambda/\rrad$ given by sets $g^n$.
Then $(\tilde{P}^n, \tilde{d}^n)$ is a minimal projective resolution for $\tilde{\Lambda}/ \tilde{\rrad}$ which is defined by sets $\tilde{g}^n$, where
\begin{enumerate}[\quad$\bullet$]
\item $\tilde{g}^0$ is the set of vertices of $\tilde{\cQ}$,
\item $\tilde{g}^1$ is the set of arrows of $\tilde{\cQ}$,
\item for $n \geqslant 2$, $\tilde{g}^n = \{\tilde{g}^n_i:=\theta^*(g^n_i) \mid g^n_i \in g^n\}$.
\end{enumerate}
\end{thm}

\begin{remark*}
Note that this agrees with the definition of $\tilde{g}^2$ given in Definition~\ref{defin:algebraconstruct}. Moreover, for $n\geqslant 2$, each $\tilde{g}^n_i$ is a uniform element which starts (resp.\ ends) at the vertex $\mathfrak{o}(g^n_i)$ (resp.\ $\mathfrak{t}(g^n_i)$) in $\mathcal{Q}_0$ and so   $\tilde{g}^n_i=\varepsilon\tilde{g}^n_i\varepsilon$.
\end{remark*}

\begin{proof}
Let $\tilde{g}^0$ be the set of vertices of $\tilde{\cQ}$, let $\tilde{g}^1$ be the set of arrows of $\tilde{\cQ}$, and let $\tilde{g}^2$ be as given in Definition~\ref{defin:algebraconstruct}.
For $n = 0, 1, 2$, define $\tilde{P}^n$ to be the projective $\tilde{\Lambda}$-module $\tilde{P}^n = \oplus_i \mt(\tilde{g}^n_i) \tilde{\Lambda}$. Define $\tilde{\Lambda}$-homomorphisms $\tilde{d}^0, \tilde{d}^1, \tilde{d}^2$ as follows:
\begin{enumerate}[\quad$\bullet$]
\item $\tilde{d}^0:\tilde{P}^0\rightarrow\tilde{\Lambda}/\tilde{\rrad}$ is the canonical surjection;
\item $\tilde{d}^1:\tilde{P}^1\rightarrow\tilde{P}^0$ is given by $\mathfrak{t}(\tilde{\alpha})\mapsto\tilde{\alpha}$ (where $\tilde{\alpha}$ is an arrow in $\tilde{\mathcal Q}$) with $\tilde{\alpha}$ in the summand of $\tilde{P}^0$ corresponding to $\mathfrak{o}(\tilde{\alpha})$;
\item write $\tilde{g}^2_i=\sum_{\tilde{\alpha}}\tilde{\alpha}\tilde{\beta}_{\tilde{\alpha}}$, where the sum is over all arrows $\tilde{\alpha}$ in $\tilde{\mathcal Q}$, and $\tilde{\beta}_{\tilde{\alpha}}\in K\tilde{\mathcal Q}$. Then $\tilde{d}^2:\tilde{P}^2\rightarrow\tilde{P}^1 $ is such that $\tilde{d}^2(\mathfrak{t}(\tilde{g}^2_i))$ has entry $\mt(\tilde{\alpha})\tilde{\beta}_{\tilde{\alpha}}$ in the summand of $\tilde{P}^1$ corresponding to $\mathfrak{t}(\tilde{\alpha})$.
\end{enumerate}
Then the sequence
$$\xymatrix@1{
\tilde{P}^2\ar[r]^{\tilde{d}^2} & \tilde{P}^1\ar[r]^{\tilde{d}^1} & \tilde{P}^0\ar[r]^{\tilde{d}^0} &
\tilde{\Lambda}/ \tilde{\rrad}\ar[r] & 0
}$$
is the first part of a minimal projective resolution of $\tilde{\Lambda}/ \tilde{\rrad}$ as defined by \cite{GSZ} so is exact.

\smallskip

Define $\tilde{g}^3 = \{\tilde{g}^3_i:=\theta^*(g^3_i) \mid g^3_i \in g^3\}$ and $\tilde{P}^3 = \oplus_i \mt(\tilde{g}^3_i) \tilde{\Lambda}$.
Fix the labelling of the set $g^2$ so that for each $g^3_i \in g^3$, there are elements $r_{i,j} \in K\cQ$ with $g^3_i = \sum_jg^2_jr_{i,j}$.
Define $\tilde{d}^3 : \tilde{P}^3\to \tilde{P}^2$ to be the $\tilde{\Lambda}$-homomorphism such that $\tilde{d}^3(\mathfrak{t}(\tilde{g}^3_i))$ has entry $\mathfrak{t}(\tilde{g}^2_j)\theta(r_{i,j})$ in the summand of $\tilde{P}^2$ corresponding to $\mathfrak{t}(\tilde{g}^2_j)$.
With these definitions, the next step is to show that the sequence
$$\xymatrix@1{
\tilde{P}^3\ar[r]^{\tilde{d}^3} & \tilde{P}^2\ar[r]^{\tilde{d}^2} & \tilde{P}^1\ar[r]^{\tilde{d}^1} & \tilde{P}^0\ar[r]^{\tilde{d}^0} &
\tilde{\Lambda}/ \tilde{\rrad}\ar[r] & 0
}$$
is exact.
We keep the following notation.
Write $g^2_j=\sum_{\alpha}\alpha\beta_{j, \alpha}$, where the sum is over all arrows $\alpha$ in $\mathcal{Q}_1$, and $\beta_{j, \alpha} \in K\cQ$. Then $\tilde{g}^2_j =
\sum_{\alpha \in {\cQ}_1}\alpha_1\alpha_2\cdots\alpha_A\theta^*(\beta_{j,\alpha})$
where $\theta^*(\alpha) = \alpha_1\alpha_2\cdots\alpha_A = \alpha_1\tilde{q}_{\mt(\alpha_1)}$.
Let $\tilde{x}\in \tilde{P}^2$ and write $\tilde{x}=\sum_j\mathfrak{t}(\tilde{g}^2_j)\tilde{\lambda}_j$ for some $\tilde{\lambda}_j\in \tilde{\Lambda}$.
Then $\tilde{d}^2(\tilde{x})$ has entry
$\sum_j\tilde{q}_{\mt(\alpha_1)}\theta(\beta_{j,\alpha})\tilde{\lambda}_j$
in the summand of $\tilde{P}^1$ corresponding to $\mathfrak{t}(\alpha_1)$.

First we show that $\Ker\tilde{d}^2\subseteq \Im \tilde{d}^3$.
Let $\tilde{x}=\sum_j\mathfrak{t}(\tilde{g}^2_j)\tilde{\lambda}_j \in \Ker\tilde{d}^2$ so that $\tilde{d}^2(\tilde{x})=0$. Then $\tilde{q}_{\mt(\alpha_1)}\sum_j\theta(\beta_{j,\alpha})\tilde{\lambda}_j=0$
for each arrow $\alpha \in {\cQ}_1$.
We have $\tilde{x}\varepsilon = \sum_j\mathfrak{t}(\tilde{g}^2_j)\tilde{\lambda}_j\varepsilon$ and $\varepsilon\tilde{\lambda}_j\varepsilon = \theta(\lambda_j)$ for some $\lambda_j \in \Lambda$. So
$0=\tilde{q}_{\mt(\alpha_1)}\sum_j\theta(\beta_{j,\alpha})\tilde{\lambda}_j\varepsilon=
\tilde{q}_{\mt(\alpha_1)}\theta(\sum_j\beta_{j,\alpha}\lambda_j)$.
Hence from Proposition~\ref{prop:theta props}(2), we have  $\theta(\sum_j\beta_{j,\alpha}\lambda_j)=0$ and so $\sum_j\beta_{j,\alpha}\lambda_j=0$ for each arrow $\alpha \in {\cQ}_1$ since $\theta$ is one-to-one.
Let $x_\varepsilon = \sum_j\mathfrak{t}(g^2_j)\lambda_j \in P^2$, so we have $x_\varepsilon\in\Ker d^2$.
But $\Im d^3 = \Ker d^2$ since $(P^n, d^n)$ is a minimal projective resolution of $\Lambda/\rrad$, so
$x_\varepsilon\in\Im d^3$.
By Proposition~\ref{prop for thm}, $\sum_j\mathfrak{t}(\tilde{g}^2_j)\theta(\lambda_j)$ is in $\Im\tilde{d}^3$, that is, $\tilde{x}\varepsilon$ is in $\Im\tilde{d}^3$.
Now let $w\in \tilde{\mathcal Q}_0\setminus {\mathcal Q}_0$. Then $\tilde{x}w = \sum_j\mathfrak{t}(\tilde{g}^2_j)\tilde{\lambda}_jw$ and $\varepsilon\tilde{\lambda}_jw = \theta(\lambda_{j,w})\tilde{p}_w$ for some $\lambda_{j,w} \in \Lambda$ by
Proposition~\ref{prop:dimension of wiL}(5). So
$0=\tilde{q}_{\mt(\alpha_1)}\sum_j\theta(\beta_{j,\alpha})\tilde{\lambda}_jw =
\tilde{q}_{\mt(\alpha_1)}\theta(\sum_j\beta_{j,\alpha}\lambda_{j,w})\tilde{p}_w$.
Using Proposition~\ref{prop:theta props}(1), we have $\theta(\sum_j\beta_{j,\alpha}\lambda_{j,w})=0$ and hence $\sum_j\beta_{j,\alpha}\lambda_{j,w}=0$ for each arrow $\alpha \in {\cQ}_1$.
Let $x_w = \sum_j\mathfrak{t}(g^2_j)\lambda_{j,w} \in P^2$, so that $x_w\in\Ker d^2=\Im d^3$.
By Proposition~\ref{prop for thm}, $\sum_j\mathfrak{t}(\tilde{g}^2_j)\theta(\lambda_{j,w})$ is in $\Im\tilde{d}^3$. Right multiplication by $\tilde{p}_w$ gives $\tilde{x}w\in \Im\tilde{d}^3$.
Thus $\tilde{x} = \tilde{x}\varepsilon + \sum_{w\in \tilde{\mathcal Q}_0\setminus {\mathcal Q}_0}\tilde{x}w\in \Im\tilde{d}^3$ and we have shown that $\Ker\tilde{d}^2\subseteq \Im \tilde{d}^3$.

Now we show that $\Im \tilde{d}^3\subseteq\Ker\tilde{d}^2$. Let $\tilde{x}= \sum_j\mt(\tilde{g}^2_j)\tilde{\lambda}_j\in \Im\tilde{d}^3$.
Then $\tilde{x}\varepsilon = \sum_j\mt(\tilde{g}^2_j)\tilde{\lambda}_j\varepsilon$
and $\varepsilon\tilde{\lambda}_j\varepsilon = \theta(\lambda_j)$ for some $\lambda_j \in \Lambda$.
So $\tilde{x}\varepsilon = \sum_j\mt(\tilde{g}^2_j)\theta(\lambda_j) \in \Im\tilde{d}^3$. Let $x_\varepsilon = \sum_j\mt(g^2_j)\lambda_j$ so that $x_\varepsilon \in \Im d^3$ by Proposition~\ref{prop for thm}.
Thus $x_\varepsilon \in \Ker d^2$ since $(P^n, d^n)$ is a minimal projective resolution of $\Lambda/\rrad$.
So $\sum_j\beta_{j,\alpha}\lambda_j=0$ for each arrow $\alpha \in {\cQ}_1$ and thus $\theta(\sum_j\beta_{j,\alpha}\lambda_j)=0$. Hence $\tilde{q}_{\mt(\alpha_1)}\sum_j\theta(\beta_{j,\alpha})\tilde{\lambda}_j\varepsilon=0$ for each arrow $\alpha \in {\cQ}_1$ and so $\tilde{d}^2(\tilde{x}\varepsilon) = 0$.
Now let $w\in \tilde{\mathcal Q}_0\setminus {\mathcal Q}_0$.
Then $\tilde{x}w = \sum_j\mathfrak{t}(\tilde{g}^2_j)\tilde{\lambda}_jw$ and $\varepsilon\tilde{\lambda}_jw = \theta(\lambda_{j,w})\tilde{p}_w$ for some $\lambda_{j,w} \in \Lambda$. So
$\tilde{x}w = \sum_j\mathfrak{t}(\tilde{g}^2_j)\theta(\lambda_{j,w})\tilde{p}_w\in \Im\tilde{d}^3$.
Let $x_w = \sum_j\mathfrak{t}(g^2_j)\lambda_{j,w}$. By Proposition~\ref{prop for thm}, $x_w \in \Im d^3$.
Thus $x_w \in \Ker d^2$ since $(P^n, d^n)$ is a minimal projective resolution of $\Lambda/\rrad$.
So $\sum_j\beta_{j,\alpha}\lambda_{j,w}=0$ for each arrow $\alpha \in {\cQ}_1$ and thus $\theta(\sum_j\beta_{j,\alpha}\lambda_{j,w})\tilde{p}_w=0$. Hence $\tilde{q}_{\mt(\alpha_1)}\sum_j\theta(\beta_{j,\alpha})\tilde{\lambda}_jw=0$ for each arrow $\alpha \in {\cQ}_1$ and so $\tilde{d}^2(\tilde{x}w) = 0$.
Since $\tilde{x} = \tilde{x}\varepsilon + \sum_{w\in \tilde{\mathcal Q}_0\setminus {\mathcal Q}_0}\tilde{x}w$, it follows that $\tilde{x}$ is in $\Ker\tilde{d}^2$ and hence $\Im \tilde{d}^3 \subseteq\Ker\tilde{d}^2$.
Therefore the sequence is exact up to $\tilde{P}^3$.

\smallskip

Finally, we recall that the functor $T_\varepsilon$ is exact and $(P^n, d^n)$ is a minimal projective resolution of $\Lambda/\rrad$ given by sets $g^n$, so the sequence
$$\xymatrix@1{
\cdots \ar[r] & T_\varepsilon(P^n)\ar[r]^{\!\! T_\varepsilon(d^n)} & T_\varepsilon(P^{n-1})\ar[r]
& \cdots\ar[r] & T_\varepsilon(P^3)\ar[r]^{T_\varepsilon(d^3)} & T_\varepsilon(P^2)
}$$
is exact and $T_\varepsilon(P^n)$ is a projective $\tilde{\Lambda}$-module for all $n \geqslant 2$. Identifying $\Lambda$ with $\varepsilon\tilde{\Lambda}\varepsilon$, we have
$T_\varepsilon(P^n) = \oplus_{i} \mt(\theta^*(g^n_i))\varepsilon\tilde{\Lambda}\varepsilon \otimes_{\varepsilon\tilde{\Lambda}\varepsilon}\varepsilon\tilde{\Lambda} \cong \oplus_{i} \mt(\theta^*(g^n_i))\tilde{\Lambda}$ for $n \geqslant 2$. In particular, we identify $T_\varepsilon(P^2)$ with $\tilde{P}^2$ and $T_\varepsilon(P^3)$ with $\tilde{P}^3$. It is easy to verify that $T_\varepsilon(d^3)$ is then identified with $\tilde{d}^3$.
So the sequence
$$\xymatrix@1{
\cdots \ar[r] & T_\varepsilon(P^n)\ar[r]^{\!\!\!\! T_\varepsilon(d^n)} & T_\varepsilon(P^{n-1})\ar[r]
& \cdots\ar[r] & T_\varepsilon(P^4)\ar[r]^{\ \ T_\varepsilon(d^4)} & \tilde{P}^3\ar[r]^{\tilde{d}^3} & \tilde{P}^2
}$$
is exact.

Hence we have a projective resolution
$$\xymatrix@1{
\cdots \ar[r] & T_\varepsilon(P^n)\ar[r]^{\ \ T_\varepsilon(d^n)} &
\ \cdots \ar[r] & T_\varepsilon(P^4)\ar[r]^{\ \ T_\varepsilon(d^4)} & \tilde{P}^3\ar[r]^{\tilde{d}^3} & \tilde{P}^2\ar[r]^{\tilde{d}^2} & \tilde{P}^1\ar[r]^{\tilde{d}^1} & \tilde{P}^0\ar[r]^{\tilde{d}^0} &
\tilde{\Lambda}/ \tilde{\rrad}\ar[r] & 0
}$$
for $\tilde{\Lambda}/ \tilde{\rrad}$.

To complete the proof, let $n \geqslant 4$, and define
$\tilde{g}^n = \{\tilde{g}^n_i:=\theta^*(g^n_i) \mid g^n_i \in g^n\}$ and $\tilde{P}^n = \oplus_i \mt(\tilde{g}^n_i) \tilde{\Lambda}$.
Write $g^n_i=\sum_jg^{n-1}_jr_j$ for some $r_j\in{K\mathcal{Q}}$, so that  $\tilde{g}^n_i=\sum_j\tilde{g}^{n-1}_j\theta^*(r_j)$.
Define $\tilde{d}^n : \tilde{P}^n \rightarrow \tilde{P}^{n-1}$ to be the $\Lambda$-homomorphism where $\tilde{d}^n(\mathfrak{t}(\tilde{g}^n_i))$ has entry $\mathfrak{t}(\tilde{g}^{n-1}_j)\theta(r_j)$ in the summand of $\tilde{P}^{n-1}$ corresponding to $\mathfrak{t}(\tilde{g}^{n-1}_j)$.
Then the identification of $T_\varepsilon(P^n)$ with $\tilde{P}^n$ (for $n\geqslant 2$) also identifies $T_\varepsilon(d^n)$ with $\tilde{d}^n$ for $n \geqslant 3$. So we have a projective resolution
$$\xymatrix@1{
\cdots \ar[r] & \tilde{P}^n\ar[r]^{\tilde{d}^n} & \tilde{P}^{n-1}\ar[r]
& \cdots\ar[r] & \tilde{P}^3\ar[r]^{\tilde{d}^3} & \tilde{P}^2\ar[r]^{\tilde{d}^2} & \tilde{P}^1\ar[r]^{\tilde{d}^1} & \tilde{P}^0\ar[r]^{\tilde{d}^0} &
\tilde{\Lambda}/ \tilde{\rrad}\ar[r] & 0
}$$
for $\tilde{\Lambda}/ \tilde{\rrad}$ given by the sets $\tilde{g}^n$. Minimality follows since $\Im\tilde{d}^n \subseteq \rad(\tilde{P}^{n-1})$ for all $n\geqslant 0$.
\end{proof}

\begin{prop}\label{prop for thm}
Let $z = \sum_j\mt(g^2_j)\mu_j \in P^2$ and $\tilde{z} = \sum_j\mt(\tilde{g}^2_j)\theta(\mu_j)\in \tilde{P}^2$, where $\mu_j \in \Lambda$.
Then the following are equivalent:
\begin{enumerate}
\item $z \in \Im d^3$;
\item $\tilde{z} \in \Im\tilde{d}^3$;
\item $\tilde{z}\tilde{p}_w \in \Im\tilde{d}^3$ for each $w \in \tilde{\mathcal Q}_0\setminus {\mathcal Q}_0$.
\end{enumerate}
\end{prop}

\begin{proof}
We keep the notation of the proof of Theorem~\ref{thm:projres}.
Let $z = \sum_j\mt(g^2_j)\mu_j\in P^2$, $y = \sum_i\mt(g^3_i)s_i\in P^3$, $\tilde{z} = \sum_j\mt(\tilde{g}^2_j)\theta(\mu_j)\in \tilde{P}^2$ and $\tilde{y} = \sum_i\mt(\tilde{g}^3_i)\theta(s_i)\in\tilde{P}^3$.
For each $g^2_j\in g^2$, consider the summand of $P^2$ (resp.\ $\tilde{P}^2$) corresponding to $\mt(g^2_j)$ (resp.\ $\mt(\tilde{g}^2_j)$).

By definition, $d^3(\mt(g^3_i))$ has entry $\mt(g^2_j)r_{i,j}$ in the summand of $P^2$ corresponding to $\mt(g^2_j)$,
and $\tilde{d}^3(\mathfrak{t}(\tilde{g}^3_i))$ has entry $\mathfrak{t}(\tilde{g}^2_j)\theta(r_{i,j})$ in the summand of $\tilde{P}^2$ corresponding to $\mathfrak{t}(\tilde{g}^2_j)$.
So $d^3(y)$ has entry $\sum_i\mt(g^2_j)r_{i,j}\mt(g^3_i)s_i$ in the summand of $P^2$ corresponding to $\mt(g^2_j)$, and $\tilde{d}^3(\tilde{y})$ has entry
$\sum_i\mt(\tilde{g}^2_j)\theta(r_{i,j})\mt(\tilde{g}^3_i)\theta(s_i)$ in the summand of $\tilde{P}^2$ corresponding to $\mt(\tilde{g}^2_j)$.
Since $\theta$ is one-to-one, $\mt(g^2_j)r_{i,j}\mt(g^3_i)s_i = \mt(g^2_j)\mu_j$ if and only if
$\mt(\tilde{g}^2_j)\theta(r_{i,j})\mt(\tilde{g}^3_i)\theta(s_i) = \mt(\tilde{g}^2_j)\theta(\mu_j)$.
Hence $z = d^3(y)$ if and only if $\tilde{z} = \tilde{d}^3(\tilde{y})$.
By Proposition~\ref{prop:theta props}(1), $\mt(\tilde{g}^2_j)\theta(r_{i,j})\mt(\tilde{g}^3_i)\theta(s_i) = \mt(\tilde{g}^2_j)\theta(\mu_j)$ if and only if $\mt(\tilde{g}^2_j)\theta(r_{i,j})\mt(\tilde{g}^3_i)\theta(s_i)\tilde{p}_w = \mt(\tilde{g}^2_j)\theta(\mu_j)\tilde{p}_w$ where $w \in \tilde{\mathcal Q}_0\setminus {\mathcal Q}_0$.
Hence $\tilde{z} = \tilde{d}^3(\tilde{y})$ if and only if $\tilde{z}\tilde{p}_w = \tilde{d}^3(\tilde{y}\tilde{p}_w)$. The result follows.
\end{proof}

The rest of this section concerns the application of Theorem~\ref{thm:projres} to $d$-Koszul algebras, whereby we recover a result of Leader \cite[Theorem 8.15]{L}.
Recall that a graded algebra $\Lambda = \Lambda_0 \oplus\Lambda_1 \oplus\cdots$ is said to be Koszul if $\Lambda_0$ has a linear resolution, that is, if the $n$th projective module $P^n$ in a minimal graded projective resolution $(P^n, d^n)$ of $\Lambda_0$ is generated in degree $n$.
Berger then introduced $d$-Koszul algebras, for $d \geqslant 2$, in \cite{B} motivated by certain cubic Artin-Schelter regular
algebras and anti-symmetrizer algebras.
For finite-dimensional algebras, these were further generalised, firstly to $(D,A)$-stacked monomial algebras by Green and Snashall (\cite[Definition 3.1]{GS-colloq math} and see \cite{GS-J Alg}) and then by Leader and Snashall to $(D, A)$-stacked algebras in \cite{LS}.

\begin{defin}\cite[Definition 2.1]{LS}\label{defin:(D,A)-stacked}
Let $\Lambda=K\mathcal{Q}/I$ be a finite-dimensional algebra. Then $\Lambda$ is a $(D, A)$-stacked
algebra if there is some $D \geqslant 2, A \geqslant 1$ such that, for all $0 \leqslant n \leqslant \gldim \Lambda$, the projective module $P^n$ in a minimal projective resolution of
$\Lambda/\rrad$ is generated in degree $\delta(n)$, where
$$\delta(n)=
\begin{cases}
0 & \mbox{if $n=0$}\\
1 & \mbox{if $n=1$}\\
\frac{n}{2}D & \mbox{if $n$ even, $n \geqslant 2$}\\
\frac{n-1}{2}D+A & \mbox{if $n$ odd, $n \geqslant 3$.}
\end{cases}$$
\end{defin}

When $A=1$ and $D=d$, the $(d,1)$-stacked algebras are precisely the finite-dimensional $d$-Koszul algebras of Berger (with the case $A=1, D=2$ giving the finite-dimensional Koszul algebras). In all cases, $\Lambda$ is a graded algebra in which the $n$th projective module $P^n$ in a minimal graded projective resolution $(P^n, d^n)$ of $\Lambda_0$ is generated in a single degree, and for which the Ext algebra $E(\Lambda)$ is finitely generated. Specifically, it was shown in \cite[Theorem 4.1]{GMMVZ}, that the Ext algebra of a $d$-Koszul algebra is generated in degrees 0, 1 and 2, and, in \cite[Theorem 2.4]{LS} that the Ext algebra of a $(D, A)$-stacked algebra is generated in degrees 0, 1, 2 and 3.

We now apply Theorem~\ref{thm:projres} to $d$-Koszul algebras.

\begin{thm}\cite[Theorem 8.15]{L}\label{(D,A)}
Let $\Lambda=K\mathcal{Q}/I$ be a $d$-Koszul algebra for some $d\geqslant 2$. Let $A\geqslant 1$ and set $D = dA$.
Then the algebra $\tilde{\Lambda}_A$ is a $(D,A)$-stacked algebra.
\end{thm}

\begin{proof}
Let $\Lambda=K\mathcal{Q}/I$ be a $d$-Koszul algebra. Let $(P^n, d^n)$ be a minimal projective resolution for $\Lambda/\rrad$ given by sets $g^n$.
Then $P^n$ is generated in degree
$$\begin{cases}
\frac{n}{2}d & \mbox{if $n$ even, $n \geqslant 0$}\\
\frac{n-1}{2}d + 1 & \mbox{if $n$ odd, $n \geqslant 1$}
\end{cases}$$
and each $g^n_i \in g^n$ is a uniform homogeneous element with
$$\ell(g^n_i) =
\begin{cases}
\frac{n}{2}d & \mbox{if $n$ even, $n \geqslant 0$}\\
\frac{n-1}{2}d + 1 & \mbox{if $n$ odd, $n \geqslant 1$.}
\end{cases}$$
Let $(\tilde{P}^n, \tilde{d}^n)$ be the minimal projective resolution for $\tilde{\Lambda}_A/ \tilde{\rrad}_A$ given by sets $\tilde{g}^n$ from Theorem~\ref{thm:projres}.
For $n \geqslant 2$, and each $g^n_i \in g^n$, we have $\ell(\tilde{g}^n_i) = A\cdot\ell(g^n_i)$. Thus
$$\ell(\tilde{g}^n_i) =
\begin{cases}
\frac{n}{2}dA & \mbox{if $n$ even, $n \geqslant 2$}\\
\frac{n-1}{2}dA + A & \mbox{if $n$ odd, $n \geqslant 3$.}
\end{cases}$$
Let $D=dA$. Then, for all $n \geqslant 0$, we have $\ell(\tilde{g}^n_i) = \delta(n)$ so $\tilde{P}^n$ is generated in degree $\delta(n)$. Thus $\tilde{\Lambda}_A$ is a $(D,A)$-stacked algebra.
\end{proof}

\begin{example}\label{example:stacked}
Let $\Lambda, \tilde{\Lambda}$ be the algebras of Example~\ref{example:fg}(2).
The algebra $\Lambda$ is a $d$-Koszul monomial algebra with $d=3$. It now follows from Theorem~\ref{(D,A)} that $\tilde{\Lambda}$ is a $(6,2)$-stacked monomial algebra. Indeed $\tilde{\Lambda}$ was given in \cite{GS-colloq math} as an example of a $(6,2)$-stacked monomial algebra.
\end{example}

\end{document}